\documentclass[11pt]{article}
\usepackage{graphicx}
\usepackage{color}
\usepackage{amsmath, amsthm, amssymb}
\usepackage[toc,page]{appendix}
\usepackage{subfigure}

\newtheorem{theorem}{Theorem}

\numberwithin{equation}{section}

\def\ve{\varepsilon}

\begin{document}

\title{\bf A singularly perturbed convection-diffusion problem  posed on  an annulus}

\author{A. F. Hegarty
\footnote{The first author acknowledges the support of MACSI, the Mathematics Applications Consortium for Science and Industry (www.macsi.ul.ie), funded by the Science Foundation Ireland Investigator Award 12/IA/1683.} \thanks{MACSI, Department of Mathematics and Statistics, University of
Limerick, Ireland.\  email: alan.hegarty@ul.ie} 
\and E.\ O'Riordan
\thanks{School of Mathematical Sciences, Dublin City
University, Dublin 9, Ireland.\ email: eugene.oriordan@dcu.ie}
}

\maketitle
\begin{abstract}
A finite difference method is constructed for a singularly perturbed convection diffusion problem posed on an annulus. The method involves combining polar coordinates, an upwind finite difference operator  and a piecewise-uniform Shishkin mesh in the radial direction. Compatibility constraints are imposed on the data in the vicinity of certain characteristic points to ensure that interior layers do not form within the annulus. A theoretical parameter-uniform error bound is established and numerical results are presented to illustrate the performance of the numerical method applied to two particular test problems. 
\vskip0.5cm

\noindent {\bf Keywords}: {Singularly perturbed,  convection-diffusion ,  Shishkin mesh , annulus}
\vskip0.25cm

\noindent {\bf Mathematics Sublect Classification (2000)}:  { 65N12, 65N15,  65N06}
\end{abstract}

\section{Introduction}

The construction of globally pointwise accurate numerical approximations to singularly perturbed elliptic problems (of the form $Lu=f$ on $\bar \Omega$)  
posed on non-rectangular domains $\bar \Omega$ is a research area that requires development.
We shall restrict our focus to problems involving  {\it inverse-monotone differential operators} $L$. That is,  for all functions $z$ in the domain of the operator $L$, if $Lz \geq 0$ at all points in the closed domain $\bar \Omega$ then $z \geq 0$ at all points in $\bar \Omega$.  This class of problems includes convection-diffusion problems of the form  \[- \ve \triangle  u + \vec{a} \nabla  u + b u = f(x,y), \ b(x,y) \geq 0, \quad (x,y) \in \ \Omega; \quad u = g, \ (x,y) \in \partial \Omega. \] 
In any discretizations of 
singularly perturbed convection-diffusion  problems,  we seek to preserve this fundamental property of the differential operator. In other words, we require that the discretization of both the domain and  of the differential operator combine  so that the system matrix (denoted here by $L^N$) is a monotone matrix. That is, for all mesh functions $Z$, if   $L^NZ \geq {\vec {0}}$ at all mesh points then $Z \geq {\vec {0}}$ at all mesh points. It is well established  that classical finite element discretizations of singularly perturbed convection-diffusion  problems lose inverse-monotonicity. There is an extensive literature on alternative finite element formulations \cite{john1} that attempt to minimize the adverse effects of losing this property of inverse-monotonicity in the discretization process. 
Given these stability difficulties with the finite element framework, we pursue our quest for discretizations that preserve inverse-monotonicity within a finite difference formulation. 

Rectangular domains are ideally suited to a computational approach, as a tensor product of one-dimensional uniform or non-uniform meshes is a simple and  obvious discretization of the domain. For some non-rectangular domains, coordinate transformations exist  so that the non-rectangular domain can be mapped onto a rectangular domain. However, in general, the Laplacian operator $\triangle \tilde  u$ in one coordinate system is mapped to a more general elliptic operator ($au_{xx} +b u_{xy}+cu_{yy}$) in an alternative coordinate system, where the general elliptic  operator  contains a mixed second order derivative \cite{ray09}. Due to the presence of different scales in the solutions of singularly perturbed problems, it is natural to use highly anistropic meshes, where aspect ratios of the form $h_x/h_y = O(\ve ^p), 1 \geq  p >0 $ are unavoidable in some subregions of the domain. However, we know of no discretization of a mixed second order partial derivative that  preserves inverse-monotonicity and does not  place a restriction on the aspect ratio of the form  $C_1 \leq h_x/h_y\leq C_2$, $C_1,C_2 =O(1)$ \cite{matus}. Due to this barrier to preserving stability, we look at particular non-rectangular domains for which a coordinate transformation (to a rectangular domain) exists, which does not generate a mixed second order derivative term.

Parameter-uniform numerical methods \cite{fhmos} are numerical methods designed to be  globally accurate in the  maximum norm and to satisfy an asymptotic error bound on the numerical solutions (which are, in this paper, the bilinear interpolants $\bar U^N$) of the form
\[
\Vert \bar U ^N -u \Vert _{\infty, \Omega} \leq CN^{-p}, \quad p >0,
\] 
where the error constant $C$ and the order of convergence $p$ are independent of the singular perturbation parameter $\ve$ and the discretization parameter $N$. 
Parameter-uniform numerical methods  can normally be categorized as either a fitted operator method or as a fitted mesh method. In the fitted operator (sometimes called exponential fitting) approach a uniform or quasi-uniform mesh $\bar \Omega ^N$ is used and the emphasis is on the design of a non-classical approximation $L^N_*$ to the differential operator $L$. These fitted finite difference operators can be generated by constructing a nodally exact difference operator $L^N_F$ for a  constant coefficient problem and extending it  to the corresponding variable coefficient problems (e.g. Il'in's scheme \cite{ilin}) or by enriching the solution space with non-polynomial basis functions (e.g. the Tailored Finite Point Method \cite{tailor1} or using correctors \cite{Temam4}). However, Shishkin \cite{shs89} established that for a class of singularly perturbed problems, whose solutions contain a characteristic  boundary  layer, no fitted operator method exists on a quasi-uniform mesh. This result led many researchers to the construction of fitted mesh methods, where classical finite difference  operators $L^N$ (such as simple upwinding) are combined with specially constructed layer-adapted meshes (such as the 
Shishkin mesh \cite{gis1,fhmos} or the Bakhvalov mesh \cite{bak}). In general, we are interested in developing numerical methods which can be adapted to solving problems with characteristic boundary or interior layers. Hence, our focus will be on the construction of a suitable fitted mesh. In passing, we note that the  option of combining a fitted operator (in the neighbourhood of a particular singularity) and a fitted mesh remains open to further investigation.

In \cite{circle,moscow} we examined the case of a convection-diffusion problem posed within a circular domain. In the current paper, motivated by the problem proposed in \cite{hemker}, we consider a problem posed on 
 an annular domain. In the numerical experiments in \cite{moscow} it was observed that the imposition of certain compatibility  constraints on the data (which were required to establish a theoretical error bound in the associated numerical analysis \cite{circle}) appeared unnecessary in practice, as the numerical experiments indicated that the numerical method appeared to be parameter-uniform even when these compatibility  constraints on the data were not imposed on particular test problems. 
However, in the case of an annular region, the character of the data at the interior characteristic points  is crucial and intrinsic to the problem. In general, interior parabolic layers will emerge from the  interior characteristic points, unless a sufficient level of compatibility  constraints are placed on the data to prevent such layers occurring.  Some preliminary numerical results illustrating parabolic interior layers appearing in the solution are given in \cite{china-BAIL}. 
In the current paper, we identify sufficient compatibility constraints on the data so that such interior layers do not appear in the solution and, in addition, so that a theoretical error bound can be established for a class of singularly perturbed problems posed on  an annulus. The construction, and subsequent numerical analysis, of a parameter-uniform numerical method for a singularly perturbed convection-diffusion problem (posed on an annulus), where the solution exhibits an interior parabolic layer, remains an open problem.

In \S 2 we define the continuous problem and identify constraints on the data (\ref{assump1}) to prevent interior layers appearing. The solution is decomposed into regular and boundary layer components. Pointwise bounds on the derivatives of these components of the solution  are established. In \S 3 the discrete problem is specified and  the associated numerical analysis is given. Some numerical results are presented in the final section. 

{\bf Notation:} Throughout this paper,  $C$  denotes a generic constant that is independent of the singular perturbation parameter $\ve$ and of all discretization parameters.
Throughout the paper, we will always use the pointwise maximum norm, which we denote by $\Vert \cdot \Vert$. Sometimes we attach a subscript $\Vert \cdot \Vert _D$, when we wish to emphasize the domain $D$ over which the maximum is being taken. Dependent variables specified in the computational domain $\Omega $ will be denoted simply by $g$ and their counterparts in the physical domain $\tilde\Omega$ will be identified by $\tilde g$. 

\section{Continuous problem}

Consider the singularly perturbed elliptic problem: Find $\tilde u$ such that
\begin{subequations}\label{cont-prob}
\begin{eqnarray} \tilde L\tilde u:=- \ve \triangle \tilde u +\tilde a(x,y) \tilde u_x =\tilde f, \ \hbox{in } \ \tilde\Omega := \{(x,y)| R^2_1 < x^2+y^2 < R^2_2 \};\\ 0 < \ve \leq 1;  \quad \tilde a > \alpha >0; \\
\tilde u=0, \quad \hbox{on } \{(x,y)|  x^2+y^2 =R^2_2 \}; \\  \tilde u= \tilde g, \quad \hbox{on } \{(x,y)|  x^2+y^2 =R^2_1 \}.
\end{eqnarray}\end{subequations}
 Assume that the data $\tilde a , \tilde f, \tilde g$ is sufficiently smooth so that $\tilde u \in C^{3,\alpha}( \overline {\tilde \Omega} )$.
The differential operator $\tilde L$ satisfies a minimum principle \cite[pg. 61]{prot}. As the problem is linear, there is no loss in generality in assuming homogeneous boundary conditions on the outer circle. Compatibility constraints will be imposed below on the data in the vicinity of the characteristic points $(0,\pm R_1)$ and $(0,\pm R_2)$.

For  problem (\ref{cont-prob}), boundary layers will typically form in the vicinity of the {\it inner outflow boundary} 
\[ \Gamma _1:=\{ (x,y) | -R_1<x <  0, x^2+y^2=R^2_1 \} \]
and  in the vicinity of the {\it outer outflow  boundary}
\[ \Gamma _2:=\{ (x,y) | x^2+y^2=R^2_2, 0< x < R_2 \}. \]
Moreover, when $f \equiv 0$,  if the inner boundary condition is such that $\tilde g(0,\pm R_1) \neq 0$  then an internal layer will appear in a neighbourhood of the region
\begin{equation}\label{InterS} S:=\{ (x,y) | 0<x<\sqrt{R^2_2 -R^2_1}, \vert y \vert = R_1 \}. \end{equation}

We also define the {\it inflow boundary} (which is  a disconnected set), as the union of the following two sets
\[
\Gamma _3 :=\{ (x,y) | x^2+y^2=R_2^2,\ -R_2 \leq x \leq 0 \} ,\quad  \Gamma _4:=\{ (x,y) | x^2+y^2=R_1^2, \ 0 < x \leq R_1 \}.
\]

 By using the stretched variables $x/\ve, y/\ve$ and the minimum principle, we can deduce 
\cite{ladura,hemkera} that the solution $\tilde u$ of problem (\ref{cont-prob}) satisfies the bounds
\begin{equation}\label{crude}
\vert \tilde u (x,y) \vert \leq \Bigl(\frac{R_2+x}{\alpha} \Bigr) \Vert \tilde f \Vert  +  \Vert \tilde g \Vert\quad \hbox{and} \quad
\Bigl \Vert \frac{ \partial ^k\tilde u}{\partial x ^i \partial y^j} \Bigr \Vert \leq C \ve ^{-i-j}, \qquad 0 \leq i+j \leq 3.
\end{equation}

We next define the {\it regular component}, which is potentially discontinuous across the two half-lines   defined in  (\ref{InterS}). Define the reduced operator (associated with the operator $\tilde L$) by
\begin{equation}\label{L0}
\tilde L_0 \tilde z:= \tilde a(x,y) \tilde z_x.
\end{equation}
The reduced solution $v_0$ is characterized by two influences: the upwind data on the {\it outer inflow boundary} $\Gamma _3$  and the data on the {\it inner inflow boundary}  $\Gamma _4$ in the wake of the inner circle. 
We begin with a definition of the upwind regular component $v^-$, given by
\begin{subequations}\label{vminus}
\begin{equation} v^- (x,y) := \bigl(\tilde v^-_0+ \ve \tilde v^-_1 +\ve ^2 \tilde v^-_2\bigr)(x,y), \qquad (x,y) \in \tilde \Omega;
\end{equation} where the subcomponents are the solutions of the following problems:
\begin{eqnarray}
\tilde  L_0\tilde v^-_0 = \tilde f, \quad (x,y) \in \tilde \Omega _3\quad
 \tilde v^-_0 =\tilde u=0,  (x,y) \in \bar \Gamma _3; \\
\tilde  L_0\tilde v^-_1 = \triangle \tilde v^-_0, \quad (x,y) \in \tilde \Omega _3\quad
 \tilde v^-_1 =0,  (x,y) \in \bar \Gamma _3;\\
 \tilde  L_\ve \tilde v^-_2 = \triangle \tilde v^-_1, \quad (x,y) \in \tilde \Omega \quad \tilde v^-_2 =0, (x,y) \in \partial \tilde \Omega .
\end{eqnarray}
\end{subequations}
Observe that the sub-components $\tilde v_0, \tilde v_1$ 
are solutions of first order problems and, hence, the level of regularity of these components is determined by certain compatibility conditions being imposed at the points $(0,\pm R_2)$. As in \cite{Temam1}, these compatibility conditions are of the form
\[
\frac{\partial ^{i+j}}{\partial x ^i \partial y ^j} f(0,\pm R_2)=0, \quad 0 \leq i+2j \leq n,
\]
where $n$ is sufficiently large so that $\tilde v^-_2 \in C^3(\bar \Omega)$. 

Next we define the downwind regular component over the wake region \[
\tilde \Omega ^+:= \{(x,y)|
x \geq \sqrt{R_1^2-y^2}, \vert y \vert < R^2_1  \}\] 
by \begin{subequations}\label{vplus}
\begin{equation} v^+(x,y) := \bigl(\tilde v^+_0+ \ve \tilde v^+_1 +\ve ^2 \tilde v^+_2\bigr)(x,y), \qquad (x,y) \in \tilde \Omega ^+;\end{equation} where it's three subcomponents 
satisfy:
\begin{eqnarray}
\tilde  L_0\tilde v^+_0 = \tilde f, \quad (x,y) \in \tilde \Omega ^+\quad
 \tilde v^+_0 =\tilde g ,  (x,y) \in  \Gamma _4; \\
\tilde  L_0\tilde v^+_1 = \triangle \tilde v^+_0, \quad (x,y) \in \tilde \Omega ^+\quad
 \tilde v^+_1 =0,  (x,y) \in \Gamma _4; \\
 \tilde  L_\ve \tilde v^+_2 = (\triangle \tilde v^+_1), \quad (x,y) \in \tilde \Omega ^+ \quad \tilde v^+_2 =0, (x,y) \in \partial \tilde \Omega ^+. 
\end{eqnarray}
\end{subequations}
Excluding the region $S$,  we define the regular component as
\begin{equation}\label{vdef}
\tilde v:= \tilde v^+, \ (x,y) \in \tilde \Omega ^+ \quad \hbox{and} \quad \tilde v:=  \tilde v^-, \ (x,y) \in \overline{\tilde \Omega}  \setminus (\tilde \Omega ^+\cup S).
\end{equation}
In general, the main component of $\tilde v$, which is the reduced solution $\tilde v_0$, will be discontinuous along $S$  as
\begin{eqnarray*}
\tilde v^-_0(x,y) &=& \int _{w=-\sqrt{R_2^2-y^2}}^x \frac{\tilde f(w,y)}{\tilde a(w,y)} \ dw, \quad -R_2 \leq x ,  \quad R_1 < \vert y \vert < R_2; \\
\tilde v^+_0(x,y) &=&  \tilde g(x,y) + \int _{w=\sqrt{R_1^2-y^2}}^x \frac{\tilde f(w,y)}{\tilde a(w,y)} \ dw, \quad x \geq \sqrt{R_1^2-y^2}, \vert y \vert < R^2_1. 
\end{eqnarray*} 
Hence, in order to have a continuous reduced solution, we would need to impose  the following compatibility  condition
\begin{equation}\label{assump2}
\tilde u(0,\pm R_1)  =\int _{w=-\sqrt{R_2^2-R_1^2}}^0 \frac{\tilde f(w,\pm R_1)}{\tilde a(w,\pm R_1)} \ dw.
\end{equation}
The arguments  in \cite{Temam1} could be applied to both $v_0^-$ and $v_0^+$ so that they are both sufficiently regular and satisfy certain 
additional constraints (along the horizontal lines $y=\pm R_1$)  to ensure  that $\tilde v_0 \in C^3(\overline{\tilde \Omega})$. 
However, in order to  establish pointwise bounds on the boundary layers present, we will also need to impose more severe  constraints on the data in neighbourhoods of these characteristic points. To complete the numerical analysis in this paper, we assume the following compatibility constraints on the data. 

{\bf Assumption} Assume that there exists $\delta _1, \delta _2$, with $0 < \delta _1 <  0.5R_1$, $0 < \delta _2\leq   R_2-R_1 $ such that 
\begin{equation}\label{assump1}
\tilde f(x,y) \equiv \tilde g(x,y) \equiv 0, \qquad  \vert R_1 \pm y \vert \leq \delta _1\quad \hbox{and} \quad \tilde f(x,y) \equiv 0, \qquad  \vert R_2 \pm y \vert \leq \delta _2.
\end{equation}

This assumption prevents interior parabolic layers emerging downwind of the characteristic points $(0,\pm R_1)$ and also implies that the reduced solution $v_0$ is smooth throughout the region. Moreover,  as $\tilde v_0$ and $\tilde v_1$ both satisfy first order problems, then they are both identically zero in the vicinity of the characteristic points. That is, 
 \[
  (\tilde v_0 +\ve  \tilde v_1)(x,y) \equiv 0 , \quad \hbox{if} \quad  \vert R_1 \pm y \vert < \delta _1\quad \hbox{or} \quad  \vert R_2 \pm y \vert < \delta _2.
\] 
We associate the following critical angles $\theta _*, \theta ^*$ with  assumption (\ref{assump1})
\[
 \sin \theta _* := 1-\frac{\delta _1}{R_1} , \quad 0 < \theta _* < \pi/2 \qquad 
\sin \theta ^* := 1-\frac{\delta _2}{R_2} , \quad 0 < \theta ^* < \pi/2 .
\]
Two boundary layer components $w^-$ and $w^+$ are defined by
\begin{subequations}\label{bndry-layers}
\begin{eqnarray}
\tilde L\tilde w ^- =\tilde L\tilde w^+ = 0, \quad \hbox{in } \quad \tilde\Omega ,\qquad \tilde w ^-  = \tilde w ^+  =0, \ \ \hbox{on } \    \Gamma _3  \cup   \Gamma _4;  \\
\tilde w ^- = \tilde g - \tilde v, \ \tilde w ^+  =0, \ \hbox{on } \quad \Gamma _1;\qquad   \tilde w ^+ =   - \tilde v, \tilde w ^-  =0,\  \hbox{on } \quad \Gamma _2
. 
\end{eqnarray}
By virtue of assumption (\ref{assump1}), the boundary layer component $\tilde w$ defined by
\begin{equation}
\tilde w:= \tilde w^-,\ x \leq 0, \qquad \qquad \tilde w:=\tilde w^+,\ x \geq 0
\end{equation}
\end{subequations}
is well defined and is a sufficiently smooth function throughout the domain.

Polar coordinates are a natural co-ordinate system to employ for this problem, where
$
x=r\cos \theta , \quad y=r\sin \theta.
$
In these  polar coordinates, the continuous problem (\ref{cont-prob}) is transformed into the problem: Find $u \in C^0(\bar \Omega) \cap C^3(\Omega), \Omega := \{ (r,0) \vert R_1 < r < R_2, 0 \leq \theta < 2\pi \}$, which is periodic in $\theta$,  such that
\begin{eqnarray*}
Lu:= -\frac{\ve }{r^2} u _{\theta \theta}  -\ve  u_{rr} +  \bigl( a(r,\theta)\cos (\theta ) -{\frac{\ve}{r}} )u _r  - \frac{a(r,\theta)\sin (\theta )}{r}  u _{\theta} = f,  \ \hbox{in } \Omega ; \\
u(R_1,\theta) =g(\theta), \   u(R_2,\theta) =0, \qquad 0 \leq \theta \leq 2\pi.
\end{eqnarray*}

In our analysis of the behaviour of the layer component $w$, we will make use of  smooth cut-off functions $ \psi _*(\theta ), \psi ^*(\theta )$, which are constructed  in the Appendix. 

\begin{theorem} Assume (\ref{assump1}). 
The solution $u$ of problem (\ref{cont-prob}) can be decomposed into the sum  $u=v+w$, where $v$ and $w$ are defined, respectively, in (\ref{vminus}, \ref{vplus}, \ref{vdef}) and 
(\ref{bndry-layers}). The derivatives of the regular component $v$ satisfy the bounds
\[
\Bigl \Vert \frac{\partial ^{i+j} v}{\partial  r^i \partial \theta ^j} \Bigr \Vert \leq C(1+ \ve ^{2-i-j}),\qquad i+j \leq 3, 
\]
and the boundary layer component $w$ satisfies
\begin{subequations}
\begin{eqnarray}\label{bnd-w}
    |w(r,\theta)| &\leq&  Ce^{\frac{\alpha  \cos (\theta )  (r-R_1)}{\ve}}  , \quad \cos \theta < 0\\
		|w(r,\theta)| &\leq&  Ce^{-\frac{\alpha   \cos (\theta )  (R_2-r)}{2\ve}} + C e^{\frac{-R_2\alpha \cos  (\theta ^*)(1-\sin (\theta ^*))}{2\ve}}, \quad \cos \theta > 0.
\end{eqnarray}
For all $i,j$ with  $1\le i+j\le 3,$ the derivatives of the boundary layer component $w$ satisfy
\begin{equation}\label{bnd-der-w}
\Bigl \Vert \frac{\partial ^{i+j} w}{\partial  r^i \partial \theta ^j} \Bigr \Vert \leq C \ve ^{-i-j} \quad \Bigl \Vert \frac{\partial ^{j} w}{ \partial \theta ^j} \Bigr \Vert\leq C(1+ \ve ^{1-j}),
\end{equation}
\end{subequations}
where the constant $C$ is independent of $\ve$.
 Moreover,  there exists  some $\mu > 1$ such that $\theta ^*< \mu \theta ^*< \frac{\pi}{2} $ and 
\begin{equation}
w(r,\theta) \equiv 0, \qquad \theta \in    [\mu \theta_*,\pi -\mu \theta_*] \cup [\pi+ \mu \theta_*,2\pi -\mu \theta_*].
\end{equation}

\end{theorem}
\begin{proof} The bounds on the regular component $v$ are established using the decompositions in (\ref{vminus}, \ref{vplus}, \ref{vdef}) and the argument in \cite{circle}. The bulk of the proof involves establishing the pointwise bounds on the boundary layer function $w$ and the proof is available in the appendix. 

\end{proof}

\section{Discrete problem and associated Error Analysis}

We discretize this problem using simple upwinding on a piecewise uniform Shishkin mesh \cite{mos,gis1} in the radial direction, with $M$ mesh elements uniformly distributed in the angular direction and $N$ mesh elements used in the radial direction to produce the mesh  
\begin{subequations}\label{fitted-mesh}
\begin{eqnarray}
\overline { \Omega }_S^{N,M}:=\{ (r_i,\theta _j) \vert 0\leq i \leq N,\ 0\leq j < M,  \};\\ 
\theta _j = iK,  \ j=0,1, \ldots, M-1, \quad K = \frac{2\pi}{M};\\
r_i = R_1+ ih,  i=0,1, \ldots, \frac{N}{4}, \\ r_i = R_1+ \sigma _*  + (i-N/4)H,  i= \frac{N}{4}+1, \ldots, \frac{3N}{4} ;\\
r_i = R_2-\sigma ^*  + (i-3N/4)h,  i= \frac{3N}{4}+1, \ldots, N ; \\
 \sigma _*:= \min \{ \frac{R_2-R_1}{4},  \frac{2\ve}{\alpha  \cos (\theta _*) }  \ln N \}; \\
 \sigma ^*:= \min \{ \frac{R_2-R_1}{4}, \frac{2\ve}{\alpha  \cos (\theta ^*) }  \ln N \}. 
\end{eqnarray}
\end{subequations}

The numerical method on the mesh (\ref{fitted-mesh}), will be of the following form\footnote{The 
finite difference operators  $D^+_r,D^-_r, D^\pm_r, \delta^2_r$ are, respectively,  defined by
\begin{eqnarray*}
D^+_rZ(r_i,\theta _j) :=\frac{Z (r_{i+1},\theta _j)-Z(r_i,\theta _j)}{r_{i+1}-r_i};\quad  D_r^-Z(r_i,\theta _j) :=\frac{Z(r_{i},\theta _j)-Z (r_{i-1},\theta _j)}{r_{i}-r_{i-1}}; \\
2(bD_r^\pm)Z := (b-\vert b \vert) D_r^+Z + (b+\vert b \vert) D_r^-Z; \quad \delta^2_r Z(r_i,\theta _j) :=\frac{D^+_rZ(r_{i},\theta _j)-D_r^-Z(r_i,\theta _j)}{(r_{i+1}-r_{i-1})/2}.\end{eqnarray*}}:
Find a periodic mesh function $U(r_i,-\theta _j)=U(r_i,2\pi -\theta _j);  \  R_1\leq r_i\leq R_2;$ such that 
\begin{subequations}\label{discrete-problem}
for the internal mesh points, where $\ R_1 < r_i <R_2, \ 0 \leq \theta _j < 2 \pi$,
\begin{eqnarray} 
-\frac{\ve}{r_i^2}  \delta ^2 _{\theta}U   -\ve  \delta ^2 _{r}U  +  (a\cos (\theta _j )-\frac{\ve}{r_i} ) D^{\pm}_r U  -\frac{a}{r_i}\sin (\theta _j )D^{\pm}_\theta U   = f;
\end{eqnarray}
and for the boundary mesh points
\begin{eqnarray} 
U(R_1,\theta _j) =u(R_1,\theta _j)=g(\theta _j),\ U(R_2,\theta _j) =u(R_2,\theta _j) =0,   \ 0 \leq \theta _j \leq 2\pi .
\end{eqnarray}
\end{subequations}
This numerical method is different to the numerical method examined in \cite{circle}, as a discretized version of the differential equation is used at all the internal mesh points.  
For the internal mesh points, where $i=1,2,\ldots, N-1, \  j=0,1,2,\ldots, M-1$, we define the associated finite difference operator $L^N_{r,\theta}$ as follows: For any mesh function $Z$ 
\begin{subequations}\label{diff-operator}
\begin{equation}
L^N_{r,\theta}Z := -\frac{\ve}{r^2_i}  \delta ^2 _{\theta}Z   -\ve  \delta ^2 _{r}Z  + \bigl( a\cos (\theta _j)-\frac{\ve}{r_i} ) D^{\pm}_r Z - \frac{a\sin (\theta _j)}{r_i}D^{\pm}_\theta Z; 
\end{equation}
and, for the boundary mesh points, we define  
\begin{eqnarray} 
L^N_{r,\theta}Z(R_1,\theta _j) := Z(R_1,\theta _j) ,  \  L^N_{r,\theta}Z(R_2,\theta _j) :=Z(R_2,\theta _j), \  j=0, 1,2,\ldots, M .
\end{eqnarray}
\end{subequations}

For periodic mesh functions, with $Z(r_i,\theta _j)=Z(r_i,2\pi +\theta _j),  \  R_1\leq r_i\leq R_2$, we have the following discrete comparison principle:

\begin{theorem} 
For any single valued  periodic mesh function $Z$, 
 if $L^N_{r,\theta}Z(r_i,\theta _j) \geq 0$, for $(r_i,\theta _j) \in \bar \Omega ^N$ then $Z(r_i,\theta _j) \geq 0, (r_i,\theta _j) \in \bar \Omega _S^{N,M}$.
\end{theorem}

\begin{proof} By checking the sign pattern of the elements in the system matrix, one will see that the system matrix is an $M$-matrix, which guarantees that the inverse matrix is a non-negative matrix. 
\end{proof}

The discrete solution $U$ can be  decomposed along the same lines as the continuous solution. The error in each component is then separately bounded. To this end, we define the discrete regular component $V$ as the solution of 
\begin{subequations}\label{disc-reg}
\begin{eqnarray}
L^N_{r,\theta} V(r_i,\theta _j)&=&f(r_i,\theta _j) , \  \quad  R_1<r_i<R_2,\ 0 < \theta _j < 2\pi; \\
 V(R_1,\theta _j) =v(R_1,\theta _j)&,&    V(R_2,\theta _j) =v(R_2,\theta _j),\quad   0 \leq \theta _j \leq 2\pi; 
\end{eqnarray}
\end{subequations}
and  the two  discrete layer components $W^-, W^+$
as the solutions of the following problems:
\begin{subequations}
\begin{eqnarray}\label{disc-sing1}
L^N_{r,\theta} W^\pm(r_i,\theta _j)&=&0,\   \quad  R_1<r_i<R_2, \ 0 \leq \theta _j < 2\pi; \\
 W^\pm(R_1,\theta _j) =w^\pm(R_1,\theta _j)&,&    
 W^\pm(R_2,\theta _j) =w^\pm(R_2,\theta _j),\   \forall \theta _j . 
\end{eqnarray}
\end{subequations}
All components are defined to be single valued periodic functions on $\overline \Omega _S^{N,M}$.
The next result establishes that the discrete boundary layer components $W^-,W^+$ are
 negligible outside of their respective boundary layer regions.

\begin{theorem} Assume (\ref{assump1}), $M=O(N)$ and $4\max \{ \sigma _*, \sigma ^* \}  < R_2 - R_1$. 
The discrete boundary layer functions  $W^-,W^+$   satisfy the bounds
\begin{subequations}\label{bnds-discreteW}
\begin{eqnarray}
 \vert  W^-(r_i,\theta _j) \vert &\leq& C\Pi _{j=1}^i (1+ \frac{\gamma _*h_j}{2\ve})^{-1};  \\
 \vert  W^+(r_i,\theta _j) \vert &\leq& C\frac{\Pi _{j=i}^N (1+ \frac{\gamma ^*h_j}{2\ve})}{\Pi _{j=1}^N (1+ \frac{\gamma ^*h_j}{2\ve})}+CM^{-1}; 
\end{eqnarray}\end{subequations}
where $h_i:=r_i-r_{i-1}, \ \gamma _*< \alpha  \cos (\theta _*) $ and  $ \gamma ^*< \alpha  \cos (\theta ^*) $.
 Moreover,  there exists  some $\mu _* > 1$ such that $\mu \theta ^*\leq  \mu _*\theta ^*<\frac{\pi}{2} $ and 
\begin{equation}
W^-(r_i,\theta _j) \equiv W^+(r_i,\theta _j) \equiv 0, \quad \forall \theta _j \in    [\mu \theta_*,\pi -\mu _*\theta_*] \cup [\pi+ \mu _*\theta_*,2\pi -\mu \theta_*].
\end{equation}
\end{theorem}
\begin{proof} (i) Let us first establish the bound on $W^+$. Consider the following discrete barrier function
\[
\psi ^*(\theta _j)Z^+(r_i) + CM^{-1}(r_i\cos \theta _j),  \ Z^+(r_i) := \frac{\Pi _{j=i}^N (1+ \frac{\gamma ^*h_j}{2\ve})}{\Pi _{j=1}^N (1+ \frac{\gamma ^* h_j}{2\ve})}; \ \gamma ^*\leq \alpha \cos (\mu \theta ^*) ,
\]
where $\psi ^*$ is the cut-off function defined in (\ref{cut-off}). For any radial mesh,  note the following
\begin{eqnarray*}
D^-_rZ^+(r_i) = \frac{\gamma ^*}{2\ve}Z(r_{i-1}), && D_r^+Z(r_i) = \frac{\gamma ^*}{2\ve}Z^+(r_{i}) = (1+ \frac{\gamma ^*h_i}{2\ve})Z^+(r_{i-1}); \\
-\ve \delta ^2_r Z^+(r_i) &=& - \frac{h_{i}(\gamma ^*)^2}{4\bar h_i\ve}Z^+(r_{i-1}) \geq -\frac{(\gamma ^*) ^2}{2\ve}Z^+(r_{i-1}); \\
-\ve \delta ^2_r Z^+(r_i) +\gamma ^*D^-_rZ^+(r_i)  &\geq& 0; \qquad D^\pm_rZ^+(r_i) >0; \qquad Z^+(R_2)=1.
\end{eqnarray*}
From the definition (\ref{cut-off}) of the cut-off function $\psi ^*$, we have that
\[
\sin \theta _j D^\pm _\theta \psi ^* < 0 \quad \hbox{and}\quad  \delta ^2 _\theta \psi ^*(\theta _j) = (\psi ^*)''(\theta _j) + CM^{-2}.
\]
For all $\theta \in [2\pi -\mu \theta _*, 2\pi) \cup [0, \mu \theta _*]$ and $\ve$ sufficiently small,  using the strict inequality $a > \alpha$ we have that
\begin{eqnarray*}
&& L^N_{r,\theta} \bigl(\psi ^* (\theta _j)Z^+(r_i) + CM^{-1}(r_i\cos \theta _j)\bigr)
\\
&&\geq \psi^* (\theta _j) (-\ve \delta ^2_r Z^+(r_i) +\alpha \cos (\theta) D^-_rZ^+(r_i) ) 
\\
&&\geq  \psi ^*(\theta _j) (-\ve \delta ^2_r Z^+(r_i) +\alpha \cos( \theta _*) D^-_rZ^+(r_i) ) \\
&&\geq  0, \quad \hbox{if} \quad \gamma ^*\leq \alpha \cos (\mu \theta ^*) .
\end{eqnarray*}

(ii) We next establish the bound on $W^-$ within the region where $\cos \theta < 0$ and $\sin \theta \geq 0$. As for the continuous boundary layer function $w^-$,  consider the following discrete barrier function
\[
\Psi _* (\theta _j)Z^-(r_i), \qquad  \hbox{where}\quad  Z^-(r_0)=1, \quad Z^-(r_i) := \Pi _{j=1}^i (1+ \frac{\gamma _*h_j}{2\ve})^{-1}; \ h_i:=r_i-r_{i-1},
\]
where $\gamma _*$ is a parameter to be specified later.  The function $\Psi _* (\theta _j)$ is  constructed as follows:
Let $\mu ^* > \mu$. We identify the angles corresponding to the mesh points
\[
A^M:= \min _j \{ \theta _j | \theta _j \geq \pi - \mu ^*\theta _* \}, \quad B^M:= \max _j \{ \theta _j | \theta _j \leq \pi - \mu \theta _* \},
\]
and assume that $M$ is sufficiently large so that $8\pi M^{-1} < (\mu ^*-\mu) \theta _*$. Then
\begin{eqnarray*}
\Psi _* (\theta _j) =0,  \ \forall \theta _j \in [\frac{\pi}{2}, A^M], \  \Psi _* (\theta _j) =1,  \ \forall \theta _j \in [B^M, \pi];\\
 L^N_{r,\theta} \Psi _* (\theta _j)  =0, \quad \theta _j \in (A^M,B^M).
\end{eqnarray*} 
Note that $L^N_{r,\theta} \bigl(\Psi _* (B^M) \bigr) \geq 0$. Hence, for any radial mesh, note the following
\begin{eqnarray*}
D^-_rZ^-(r_i) &=& -\frac{\gamma _*}{2\ve}Z^-(r_{i}) = -\frac{\gamma _*}{2\ve}(1+ \frac{\gamma _*h_{i+1}}{2\ve})Z^-(r_{i+1}), \\ D_r^+Z^-(r_i) &=& -\frac{\gamma_*}{2\ve}Z^-(r_{i+1}); \\
-\ve \delta ^2_r Z^-(r_i) &=& - \frac{h_{i+1}}{2\bar h_i} \frac{\gamma _*^2}{2\ve}Z^-(r_{i+1}) \geq -\frac{\gamma _*^2}{2\ve}Z^-(r_{i+1}); \\
-\ve \delta ^2_r Z^-(r_i) -\gamma _*D_r^+Z^-(r_i)  &\geq& 0; \qquad D_r^\pm Z^-(r_i) <0; \qquad Z^-(R_1)=1.
\end{eqnarray*}
For all $\theta \in (\pi - \mu ^*\theta _* , \pi) $,  assuming $\ve$ sufficiently small, and using the strict inequality $a > \alpha$ we have that,  
\begin{eqnarray*}
L^N_{r,\theta} \bigl(\Psi _* (\theta _j)Z^-(r_i) \bigr)
&=& \Psi _* (\theta _j)L^N_{r,\theta} \bigl(Z^-(r_i) \bigr) + Z^-(r_i)L^N_{r,\theta} \bigl(\Psi _* (\theta _j) \bigr)\\
&\geq& \Psi _*(\theta _j) (-\ve \delta ^2_r  +a \cos (\theta ) D^+_r ) Z^-(r_i) \\
&\geq& \Psi _* (\theta _j) (-\ve \delta ^2_r Z^-(r_i) +\alpha \cos ( \theta _*) D^+_rZ^-(r_i) ) \\
&\geq& 0, \qquad \hbox{if} \quad \gamma _* \leq \alpha  \cos ( \theta _*)  .
\end{eqnarray*}
Hence, $\vert W^-(r_i,\theta_j) \vert \leq C \Psi _* (\theta _j)Z^-(r_i)$. 
\end{proof}

\begin{theorem}\label{main} Assume the data satisfy (\ref{assump1}) and that $M=O(N)$, then
\[
 \Vert u-\bar U \Vert _{\bar \Omega} \leq C( N^{-1}+M^{-1})  (\ln N)^2 ,
\]
where $u$ is the solution of the continuous problem  and 
$\bar U$ is the bilinear interpolant of the discrete solution $U$, generated by the finite difference operator  on the piecewise-uniform mesh. 
\end{theorem}

\begin{proof}
Let $E:=U-u$ denote the pointwise error. Let us consider the truncation error at all the interior points.
At the transition point $r_i=R_1+\sigma_*, R_2-\sigma ^*$ and for $\theta \in (\mu \theta _*, \pi - \mu \theta _*) \cup (\pi + \mu \theta _*, 2\pi - \mu \theta _*)$, we have 
\[
(\alpha r_i\cos \theta _j -\ve ) <0, \ \hbox{if} \ \cos \theta < 0, \quad \hbox{and} \quad (\alpha r_i\cos \theta _j -\ve ) >0, \ \hbox{if} \ \cos \theta > 0, \]
for $\ve$ sufficiently small.
Hence at each interior mesh point $(r_i,\theta _j)$, we have the truncation error bounds
 \begin{eqnarray*}
&&\vert L^N_{r,\theta} (U-u) (r_i,\theta _j) \vert  = \vert (L_{r,\theta} -L^N_{r,\theta} ) u(r_i,\theta _j)\vert \\
&\leq& C K\ve \Bigl\Vert \frac{\partial ^3 u}{\partial \theta ^3}\Bigr\Vert  + C \ve h_i \Bigl\Vert \frac{\partial ^3 u}{\partial r ^3}\Bigr\Vert + 
 C\min \{ h_i, h_{i+1} \} \Bigl\Vert \frac{\partial ^2 u}{\partial r ^2}\Bigr\Vert + CK  \Vert \frac{\partial ^2 u}{\partial \theta ^2}\Vert.
\end{eqnarray*}
We consider only the case where $\ve$ is sufficiently small so that \[
4\max \{ \sigma _*, \sigma ^* \}  < \min \{ R_2 - R_1, 4R_2 (1- \sin (\theta ^*)) \}, \] as the  alternative case  is easily dealt with by using a classical stability and consistency argument across the entire mesh.

For the regular component,   observe that  we have the following truncation error bounds:
\begin{eqnarray*}
\vert L^N_{r,\theta} (V-v) \vert &\leq& C(N^{-1}  +M^{-1}),\qquad  R_1 < r_i <R_2, 0 \leq \theta _j < 2\pi  .
\end{eqnarray*}
Note that, since $D_\theta^\pm \cos \theta _j = - \sin \theta _j +CK$, we have that
\[
L^N_{r,\theta} (r_i\cos \theta _j) = a(r_i,\theta _j) +O(K) \geq \alpha /2
\]
and, hence,  we can use the discrete barrier function 
\[
C(N^{-1}  +M^{-1})(R_2 +r_i\cos \theta _j)   
\]
to bound the error in the regular component,. 

 Note that $w^-=W^-\equiv 0$, for $\cos \theta \geq 0$ and so we consider the error $w^--W^-$ in approximating the layer component only in the region where $\cos \theta < 0$.  
For $r_i \geq R_1+\sigma _*$, we use the pointwise bounds (\ref{bnd-w}), (\ref{bnds-discreteW}), on the continuous and discrete layer functions, and the argument in \cite[pg.72]{mos} to deduce that
\[
\vert W^--w^- \vert \leq \vert W^-\vert + \vert w^- \vert \leq CN^{-1} , \quad r_i \geq R_1+\sigma _*,\  \frac{\pi }{2} \leq \theta _j \leq \frac{3\pi }{2}.
\] Within the fine mesh, we have the truncation error bound
\begin{eqnarray*}
\vert L^N_{r,\theta} (W^--w^-) \vert \leq  C\frac{N^{-1} \ln N +M^{-1}}{\ve},\quad R_1+\sigma _* > r_i >R_1,\ \frac{\pi }{2} \leq \theta _j \leq \frac{3\pi }{2} \\
(w^--W^-)(r_i, \frac{\pi }{2})=(w^--W^-)(r_i, \frac{3\pi }{2})  =0.
\end{eqnarray*}
Note that
\[
L^N_{r,\theta} \cos \theta \geq \frac{a \sin ^2 \theta _j}{r_i} + O(\ve) +O(K).
\]
Hence, to complete the argument, we use the barrier function
\[
 C(r_i-(R_1-\sigma _*) )(\cos \theta) \frac{(N^{-1} \ln N +M^{-1})}{\ve} +CN^{-1}. 
\]
Finally, we consider the error $W^+-w^+$. 
Away from the outer boundary layer, and  where $\ve $ is sufficiently small, we observe that
\[
e^{\frac{-R_2\alpha \cos  (\theta ^*)(1-\sin (\theta ^*))}{2\ve}}\leq    C N^{-1}.
\]
Proceed as for the other boundary layer function.  The  global error bound follows as in \cite[Theorem 4]{circle}.
\end{proof}

\section{Numerical Results} 

In this final section, we examine the performance of the numerical applied to two sample problems. In both cases, the exact solutions is not known and we estimate both the errors and the rates of convergence using the double-mesh method \cite{fhmos}. 
We compute the maximum pointwise global two--mesh differences $\bar D^N_\ve$ and from these values the parameter--uniform maximum global pointwise two--mesh differences
		$\bar D^N$, defined respectively, as follows
    \[
    \bar D^N_\ve := ||\overline{U}^N-{\overline U}^{2N} ||_{\Omega^N \cup \Omega^{2N}, \infty}, \qquad \bar D^N:= \max_{\ve \in R _{\ve}} \bar D^N_\ve, \quad  R_\ve := \{ 2^{-j}: j=0,1, \ldots 20 \}; 
     \]
 where ${\overline U}^{N}$ is the bilinear interpolant of $U^{N}$, which is the numerical solution computed on the mesh $\Omega ^N$. 
Approximations $\bar p^N_\ve$ to the global  order of convergence  
and, for any particular value of  $N$, approximations to the parameter--uniform order of
global convergence $\bar p^N$ are defined, respectively,  by
\[
   \bar p^N_\ve := \log _2 \frac{\bar D_\ve^N}{\bar D_\ve ^{2N}} \quad\hbox{and} \quad   \bar p^N:= \log _2 \frac{\bar D^N}{ \bar D^{2N}}. 
  \]

{\bf Example 1}
In order to satisfy the main assumption (\ref{assump1}), we introduce the piecewise quadratic function
\[
Q(y):= \qquad  \begin{array}{cccc}  \frac{4}{(y-(R_1+\delta))(R_2-\delta -y)}{(R_2-R_1)^2}, & \hbox{for} & y \in (R_1+\delta, R_2 -\delta) \\ \\
 \frac{((R_1-\delta)^2-y^2)}{R_1^2},& \hbox{for} & y \in (-R_1+\delta, R_1 -\delta) \\ \\
\frac{4}{(y+(R_2-\delta))(R_1+\delta +y)}{(R_2-R_1)^2},& \hbox{for} &  y \in (-R_2+\delta, -R_1 -\delta) \\ \\
0 & \hbox{otherwise}  & 
\end{array}.
\]
Then we  consider  problem (\ref{cont-prob}), where
\begin{subequations}\label{ex2}
\begin{eqnarray}
R_1=1, R_2=4, \quad a(x,y) =1+ \frac{x^2 y^2}{16} \geq 1, \quad g \equiv 0; \\
f(x,y) = (1+x^2) (Q(y))^2, \quad \delta =0.2. 
\end{eqnarray}
\end{subequations}
For this particular example, the Shishkin transition points are taken to be
\[
\sigma _*= \min \{0.75 , \frac{2R_1\ve \ln N}{\sqrt{\delta (2R_1-\delta)}}   \} \quad \hbox{and} \quad   \sigma ^*
=\min \{0.75 , \frac{2R_1\ve \ln N}{\sqrt{\delta (2R_2-\delta)}} \}.
\]

\begin{table}[ht!]
\centering\small
\begin{tabular}{|c| c c c c c c c|}
\hline
\multicolumn{8}{|c|}{$p^{N}_\varepsilon$}\\[3pt]
\hline
$\bf{\varepsilon}\backslash N$&\bf{8}&\bf{16}&\bf{32}&\bf{64}&\bf{128}&\bf{256}&\bf{512}\\[3pt]
\hline 
$\bf{2^{-0}}$&1.5036    & 2.3788   &  1.4118   &  1.2447  &   1.1511  &   1.0833    & 1.0426\\
 $\bf{2^{-2}}$   &1.2593  &   1.2347  &   0.8853  &   0.9619    & 0.9895  &   0.9926  &   0.9968\\
 $\bf{2^{-4}}$   &0.8712  &   0.6946  &   0.5298   &  0.7706   &  0.8714   &  0.9270   &  0.9637\\
 $\bf{2^{-6}}$   &0.3943   &  1.2135   &  0.3763    & 0.5286    & 0.6492  &   0.7791   &  0.8467\\
  $\bf{2^{-8}}$&  0.2447    & 1.3289    & 0.2918   &  0.4527    & 0.5838   &  0.7637   &  0.8369\\
  $\bf{2^{-10}}$&   0.2026  &   1.3220   &  0.3055  &   0.4185    & 0.5776  &   0.7556  &   0.8354\\
  $\bf{2^{-12}}$&  0.1920    & 1.3027  &   0.3261  &   0.4095  &   0.5765   &  0.7542   &  0.8346\\
   $\bf{2^{-14}}$&  0.1894   &  1.2974    & 0.3316   &  0.4072   &  0.5763  &   0.7529   &  0.8352\\
   $\bf{2^{-16}}$& 0.1887  &   1.2961  &   0.3330   &  0.4066  &   0.5762  &   0.7525  &   0.8353\\
   $\bf{2^{-18}}$& 0.1886   &  1.2957   &  0.3334   &  0.4065   &  0.5762  &   0.7525   &  0.8354\\
   $\bf{2^{-20}}$& 0.1885   &  1.2957  &   0.3335  &   0.4064  &   0.5762   &  0.7524   &  0.8354\\
 \hline
$\bf{p^{N}}$& 0.1885 &   1.2957  &  0.3335   &  0.4064  &   0.5762  &   0.7524   &  0.8354\\
\hline
\end{tabular}
\caption{Computed double-mesh global orders for \eqref{ex2} for some sample values of ($N$,$\varepsilon$).}
\label{Tab2}
\normalsize
\end{table}
A plot of a typical computed solution and the associated approximate error are given in Figure \ref{fig:1} and Figure \ref{fig:2}. Boundary layers are visible at all parts of the outflow boundary. 
The global orders of convergence, given in Table 1, indicate that the method is parameter-uniform for this problem.

\begin{figure*}
  \includegraphics[width=0.8\textwidth]{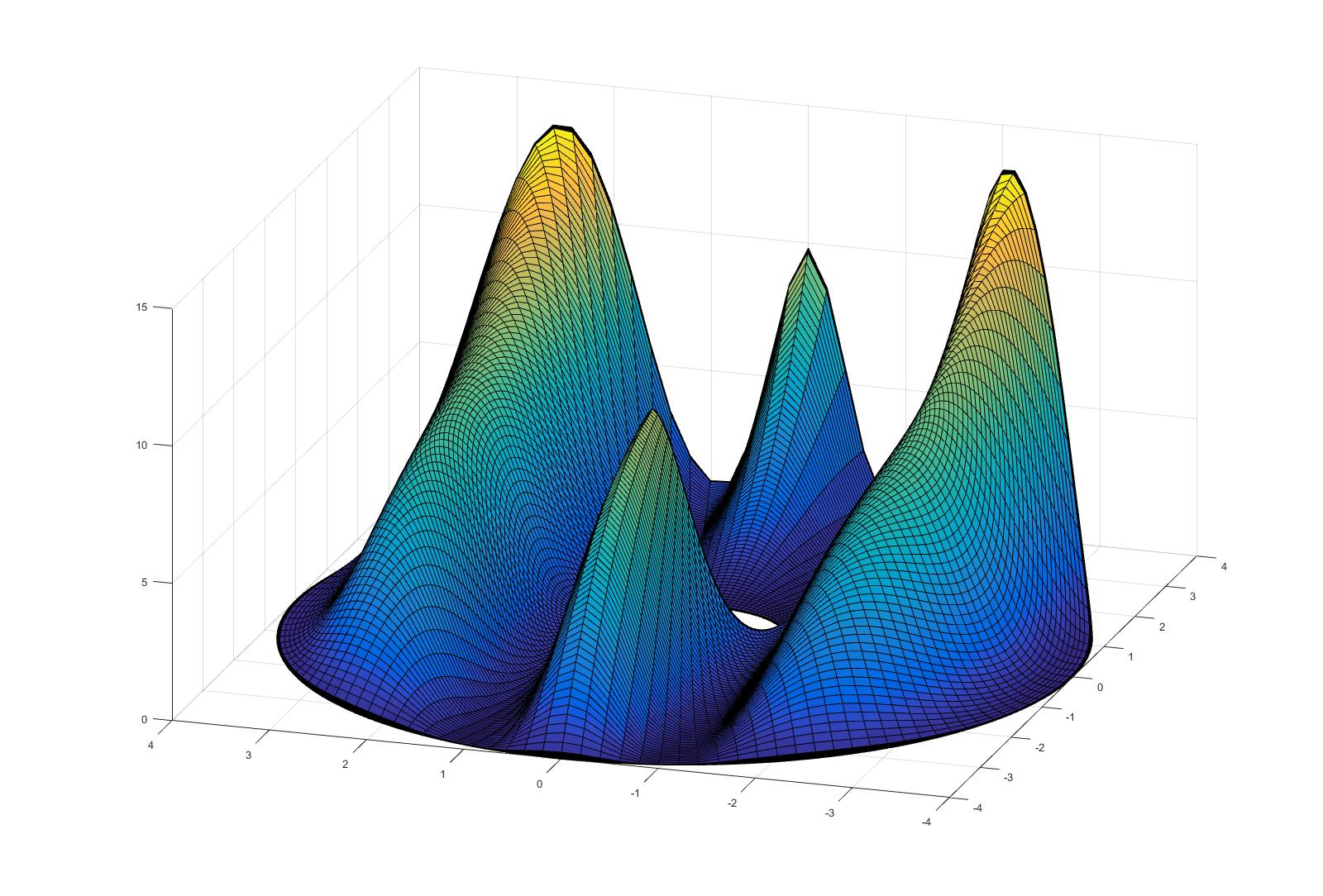}
\caption{Computed solution $\bar U^{128} $  for Example \eqref{ex2}  with $\varepsilon=2^{-10}$}
\label{fig:1}       
\end{figure*}

\begin{figure*}
 \includegraphics[width=0.8\textwidth]{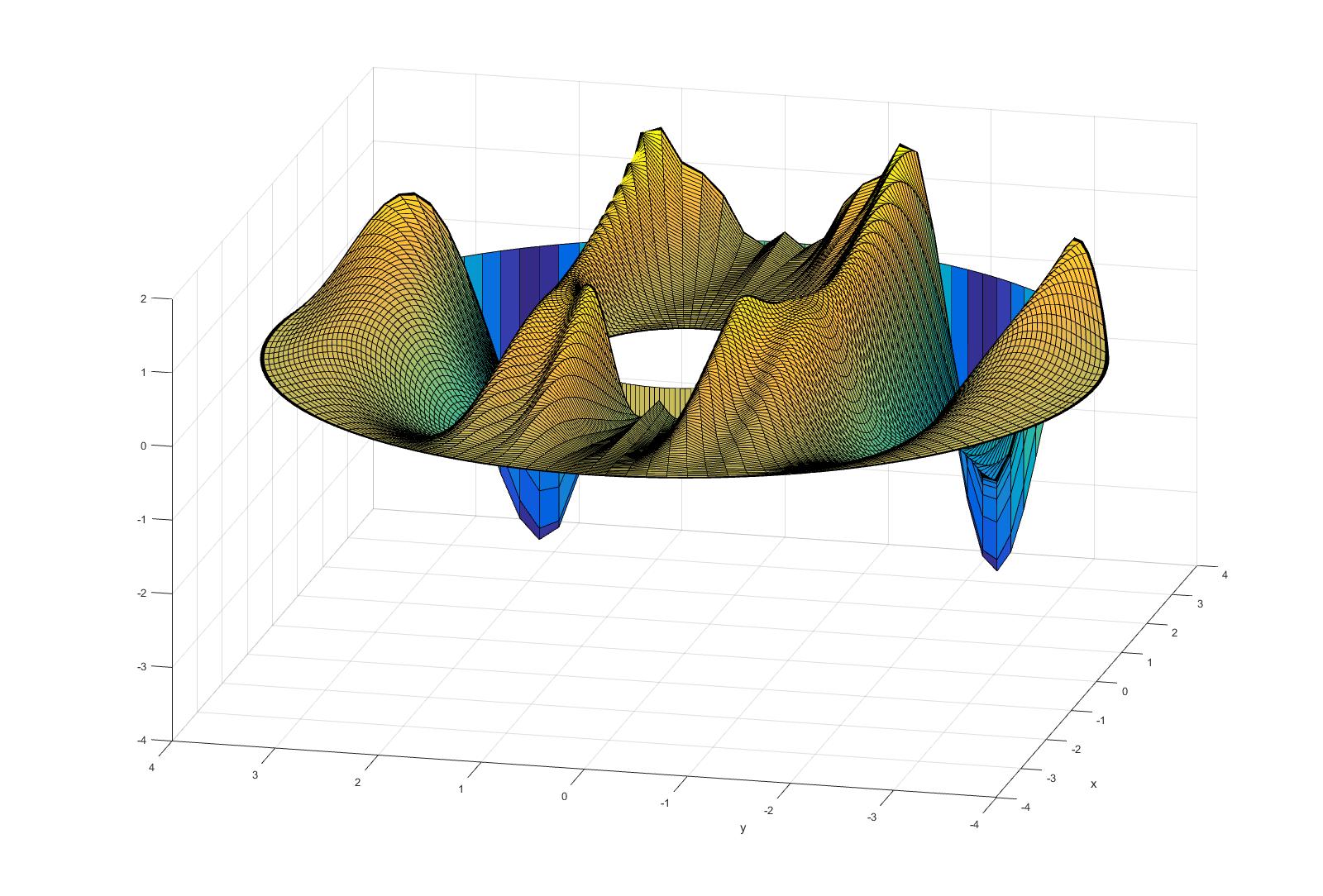}
\caption{Approximate error $\bar U^{128} - \bar U^{1024}$ for Example \eqref{ex2}  with $\varepsilon=2^{-10}$}
\label{fig:2}       
\end{figure*}

The construction of a Shishkin mesh (\ref{fitted-mesh}) is motivated by simplicity and  the objective to be parameter-uniform for a class of problems of the 
form (\ref{cont-prob}). From the pointwise upper bound on the layer component (\ref{bnd-w}) we see that the widths of the boundary layers vary with the angle $\theta$. The layer is the most thin when $\theta =0$. The mesh (\ref{fitted-mesh}) is designed  so as to encompass all angles where the boundary layer is expected to be non-zero and hence the mesh is linked to the widest angle $\theta _*$ and $ \theta ^*$ of the relevant boundary layers. In the next example we construct  a test problem, which only has an outer boundary layer, with the maximum amplitude occurring at $\theta =0$.  Moreover, the fitted mesh located around the inner boundary is not required for such a problem. Nevertheless, we have not optimized the mesh to this particular problem, as we are interested in the performance of the numerical method for a class of problems. 

{\bf Example 2}
Consider  problem (\ref{cont-prob}), with the particular choices of 
\begin{subequations}\label{ex3}
\begin{eqnarray}
R_1=1, R_2=4, \ a(x,y) =(1+0.125y^2)(1.5-0.25x) \geq 0.5, \\ f \equiv 0, \quad g(\theta)=g(2\pi-\theta), \\ 
g(\theta) = \Bigl(1-\frac{4\theta}{\pi}  \Bigr)^3\bigl(\frac{96\theta ^2}{\pi ^2}  + \frac{12 \theta }{\pi}  +1\bigr),\ \theta \in (0,\frac{\pi}{4}), \ g\equiv 0, \ \hbox{otherwise} .
\end{eqnarray}\end{subequations}
In this problem, $ g \in C^2(\Gamma _4), \ f \equiv 0$ and  the reduced solution $\tilde v_0 ^-\equiv 0$.  In this particular case,  the outer boundary layer  will only be significant when  $\vert y \vert \leq R_1 \sin \theta _*, \theta _*=\pi/4$. Hence the layer width at the outer outflow boundary will be determined by
\[
\cos \theta ^* = \sqrt{\frac{R^2_2- R_1^2 \sin ^2 \theta _*}{R^2_2}}.
\]
Hence, for this particular problem  the Shishkin transition points are taken to be 
\[
\sigma _*= \min \{0.75 , 4\sqrt{2}  \ve \ln N \} \quad \hbox{and} \quad   \sigma ^*
=\min \{0.75 , \frac{16 \ve \ln N}{\sqrt{15.5}} \}.
\]
  
	\begin{table}[ht!]
\centering\small
\begin{tabular}{|c| c c c c c c c|}
\hline
\multicolumn{8}{|c|}{$p^{N}_\varepsilon$}\\[3pt]
\hline
$\bf{\varepsilon}\backslash N$&\bf{8}&\bf{16}&\bf{32}&\bf{64}&\bf{128}&\bf{256}&\bf{512}\\[3pt]
\hline
 $\bf{2^{-0}}$&0.3183  &  2.0365   & 1.6967  &  1.5826  &  0.9896  &  0.9934 &   0.9968\\
  $\bf{2^{-2}}$&  0.4879  &  1.4344  &  0.9647  &  0.9442  &  1.0054 &   1.0115 &   1.0049\\
  $\bf{2^{-4}}$&  0.3538 &   0.2261 &   0.3908  &  0.6759 &   0.9894  &  0.9307  &  0.9764\\
  $\bf{2^{-6}}$&  0.3796 &   0.1162&    0.1830  &  0.2790  &  0.6024  &  0.8788  &  0.7976\\
  $\bf{2^{-8}}$&  0.2343  &  0.0858  &  0.2023 &   0.3064 &   0.6111 &   0.8883 &   0.7993\\
   $\bf{2^{-10}}$& 0.1933 &   0.0762  &  0.2099 &   0.3146 &   0.6226  &  0.8826  &  0.7912\\
  $\bf{2^{-12}}$&  0.1829 &   0.0737  &  0.2128  &  0.3159  &  0.6274 &   0.8789 &   0.7854\\
    $\bf{2^{-14}}$&  0.1803 &   0.0731  &  0.2136  &  0.3160  &  0.6288  &  0.8776  &  0.7839\\
     $\bf{2^{-16}}$&   0.1796  &  0.0729 &   0.2138 &   0.3161 &   0.6291  &  0.8773  &  0.7835\\
    $\bf{2^{-18}}$&  0.1794 &   0.0729 &   0.2139  &  0.3161  &  0.6292  &  0.8772 &   0.7834\\
			    $\bf{2^{-20}}$&    0.1794  &  0.0729 &   0.2139  &  0.3161  &  0.6292 &   0.8771  &  0.7834\\
 \hline 
  $\bf{p^{N}}$&  0.1794  &  0.0729  &  0.2139  &  0.3161 &   0.6292  &  0.8771 &   0.7834\\
\hline
\end{tabular}
\caption{Computed double-mesh global orders for \eqref{ex3} for some sample values of ($N$,$\varepsilon$)}
\label{Tab3}
\normalsize
\end{table}

\begin{figure}
\includegraphics[width=0.8\textwidth]{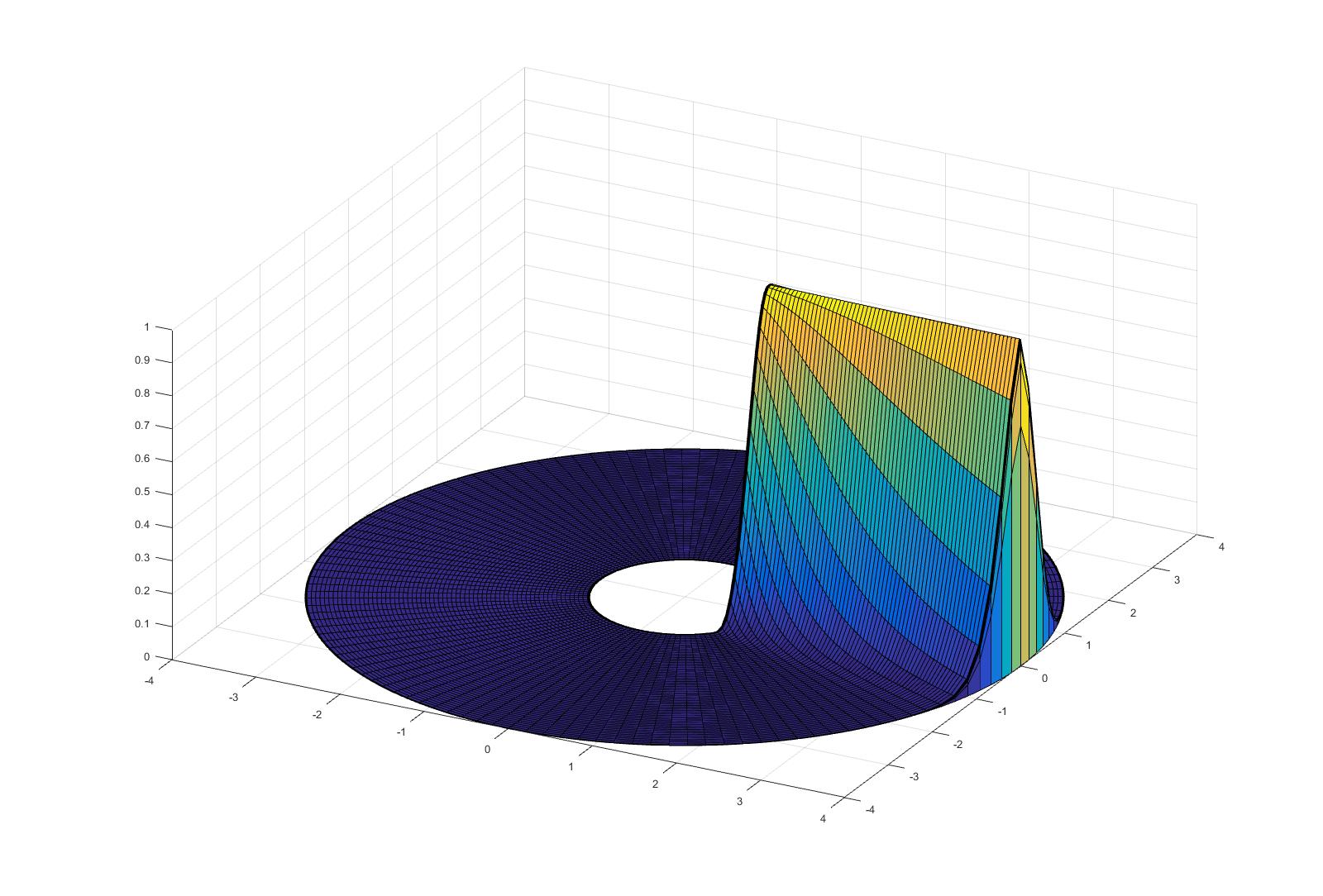}
\caption{Computed solution $\bar U^{128} $  for Example \eqref{ex3}  with $\varepsilon=2^{-10}$}
\label{fig:3}       
\end{figure}

\begin{figure}
\includegraphics[width=0.8\textwidth]{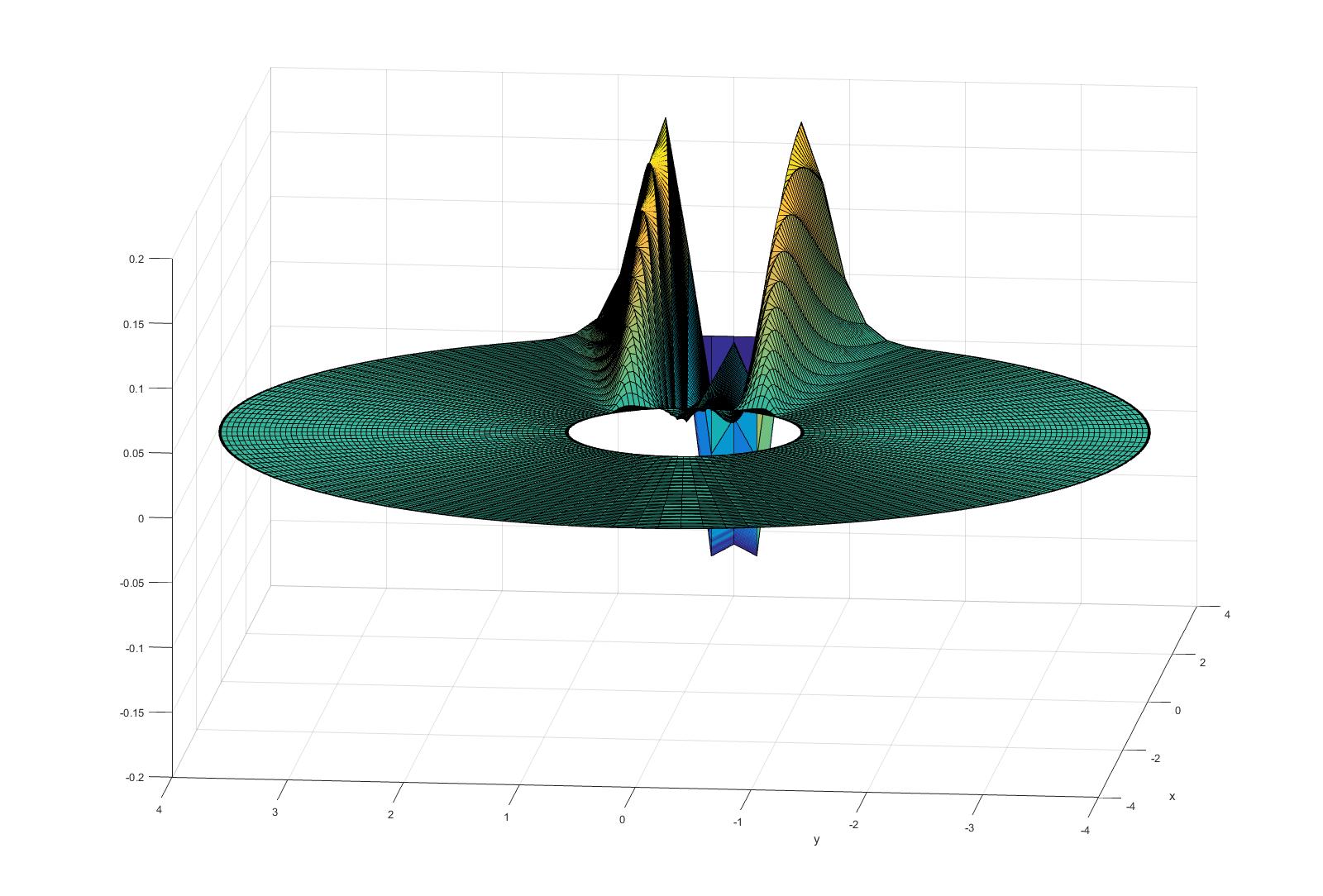}
\caption{Approximate error $\bar U^{128} - \bar U^{1024}$ for Example \eqref{ex3}   with $\varepsilon=2^{-10}$}
\label{fig:4}       
\end{figure}

   A plot of the approximate error in Figure  \ref{fig:4} demonstrates that the largest error is occurring at the outflow.    
 The global orders of convergence presented in both Tables 1 and 2 are in line with 
the theoretical error bound established in Theorem 4.

\section{Appendix: Proof of Theorem 1} 

(i) Let us first consider the upwind  region, where $\cos \theta <0$. We  construct a smooth non-negative cut-off function \[ \psi _*(\theta );[0,2\pi] \rightarrow [0,1] \] with the following properties
\begin{subequations}\label{cut-off-before}
\begin{eqnarray}
0 \leq \psi _*(\theta ) \leq 1,\quad && \psi _*(\pi-\theta ) =\psi _*(\pi+ \theta ), \ \theta \in [0,\pi];\\
\psi _* \equiv 1,\ &&  \hbox{for all}\ \theta \in    [\pi-\theta_*,\pi], \\ \psi  _* \equiv 0, \ && \theta \in   [0,\pi - \mu \theta_*], \ \theta _* < \mu \theta _* < \pi/2 ;\\
\psi _ * '(\theta ) \geq 0,&&  \hbox{for all} \ \theta \in [0,\pi/2].
\end{eqnarray}
We confine the discussion to the region where $\sin \theta  \geq 0$. By symmetry, an analogous analysis can be performed in the region where $\sin \theta <0$.  
In addition to the above properties, the cut-off function will be constructed so that it rapidly changes from the value of $0$ to $1$ over the subinterval 
$[\pi - \mu \theta_*, \pi -  \theta_*]$. To construct this function, we  initially define an associated function $v$,  as the solution of the following singularly perturbed boundary value problem 
\begin{eqnarray*}
- \ve ^2 v'' - \ve  R_1 \Vert a \Vert v' +C_* v &=&0, \quad \pi - \mu _2 \theta _* < \theta < \pi - \mu _1 \theta _* ;\quad  1 < \mu _1 < \mu _2 < \mu;\\
v (\theta ) =0, \ 0 \leq \theta \leq \pi - \mu _2 \theta _*&&  v(\theta ) =1 , \pi - \mu _1 \theta _* \leq \theta \leq \pi;  
\end{eqnarray*}  
where $C_*$ is a positive constant, which is constrained by an upper bound below.  In passing, we note that this function $v$ has a boundary layer to the left of $\theta = \pi - \mu _1 \theta _*$. 
By using the maximum principle, we see that $0 < v \leq 1, v' >0$.

 Finally, we can now specify the cut-off function, with all of the required properties:
\[
\psi _* (\theta) : = \int _{y=\theta -\mu _3 \theta _*}^{\theta +\mu _3 \theta _*}  \eta (\theta -y) v(y) \ dy = \int _{s= -\mu _3 \theta _*}^{+\mu _3 \theta _*}  \eta (s) v(\theta -s) \ ds;
\]
where $\mu _3$ is chosen sufficiently small so that $1 < \mu _1 - \mu _3 < \mu _1 + \mu _3 < \mu _2 - \mu _3 < \mu _2 + \mu _3 < \mu$ and the function $ \eta$ is a mollifier {\footnote{For example we could have
\[
\eta (s) = C e^{-\frac{(\mu _3 \theta _*)^2}{(\mu _3 \theta _*)^2 -\vert s \vert^2} },\ s \in (-\mu _3 \theta _*, \mu _3 \theta _*), \qquad \int _{s=-\mu _3 \theta _*}^{\mu _3 \theta _*} \eta (s) ds =1.
\]}
 } with support $[-\mu _3 \theta _*, \mu _3 \theta _*]$. Note also that
\begin{eqnarray}
- \ve ^2 \psi _*'' - \ve  R_1 \Vert a \Vert \psi _*' +C_* \psi _* &=&0, \quad \pi -  \mu \theta _*   < \theta < \pi - \theta _* ; \\
\psi _* (\theta ) =0, \ 0 \leq \theta \leq \pi -  \mu\theta _* &&  \psi _*(\theta ) =1 , \pi - \theta _*  \leq \theta \leq \pi.
\end{eqnarray}
\end{subequations}
Let us now consider the following  potential barrier function 
\begin{equation}\label{barrier}
B^-(r,\theta ) := \psi _*(\theta) E_-(r,\theta),  \quad E_-(r,\theta) :=e^{\frac{\alpha \kappa \cos (\theta ) (r-R_1)}{\ve}},\ 0< \kappa < 1,
\end{equation}
where $\kappa$ is arbitrary. 
Note the following expressions for the partial derivatives of the barrier function $B^-$:
\begin{eqnarray*}
B^-_\theta &=& -\frac{\alpha \kappa \sin (\theta ) }{\ve} (r-R_1) B^- + \psi _*'(\theta) E_-; \\
B^-_{\theta \theta} 
&=& -\bigl( \alpha \frac{\kappa \cos (\theta) }{\ve} (r-R_1) - \frac{\alpha ^2 \kappa ^2 \sin ^2 (\theta) }{\ve ^2} (r-R_1)^2 \bigr) B^- \\
&& \qquad + (\psi _*''(\theta) -2\psi _*'(\theta ) \frac{\alpha \kappa \sin(\theta) }{\ve} (r-R_1)) 
E_-; \\
B^-_r &=& \frac{\alpha \kappa \cos (\theta) }{\ve}  B^-,   \qquad
B^-_{rr} = \frac{\alpha ^2\kappa ^2 \cos ^2 (\theta) }{\ve ^2}  B^-.
\end{eqnarray*}
Combining these expressions for the derivatives, we deduce that  
\begin{eqnarray*}
&&LB^-= -\frac{\ve}{r^2} B^-_{\theta \theta} -\ve B^-_{rr}+  (\alpha\cos  (\theta)  -\frac{\ve}{r})B^-_r -\alpha \frac{\sin (\theta)}{r} B^-_\theta
 = \frac{\alpha ^2\kappa B^-}{\ve}T_1 + T_2 E_-; \\ \\
\hbox{where} \quad T_1
&& =
\sin ^2 (\theta) (1- \frac{R_1}{r}) (1-\kappa (1- \frac{R_1}{r})) + \cos ^2 (\theta) (1-\kappa ) -\ve \frac{R_1}{\alpha r^2} \vert \cos \theta \vert
\\
&&\geq   \cos ^2 (\theta) (1-\kappa ) -\ve \frac{\vert \cos \theta\vert }{\alpha R_1} \\ \\
\hbox{and} &&T_2 \geq \frac{-\ve \psi _*'' - \alpha R_1 \psi _* '}{R_1^2}. 
\end{eqnarray*}
Hence, the proposed barrier function $B^-$ satisfies the inequality
\begin{eqnarray*}
-\frac{\ve}{r^2} B^-_{\theta \theta} -\ve B^-_{rr}+  (\alpha\cos  (\theta)  -\frac{\ve}{r})B^-_r -\alpha \frac{\sin (\theta)}{r} B^-_\theta 
\geq
\frac{-\ve ^2\psi _*'' - \alpha R_1 \ve \psi _* '+ \bigl(\alpha R_1 \cos \theta (1-\kappa)\bigr) ^2\kappa  \psi _*}{\ve R_1^2}   E_-.
\end{eqnarray*}
In addition, this potential barrier function also satisfies
\begin{eqnarray*}
\cos \theta B^-_r- \frac{\sin \theta}{r}B^-_\theta & \geq&
 \frac{ \alpha \kappa}{\ve }(\cos ^2 \theta  + \frac{(r-R_1) \sin ^2 \theta }{r}  ) B^- -  \frac{  \sin \theta }{r }\psi_*' E_-\\
& \geq& \frac{ \alpha \kappa}{\ve }(1  - \frac{R_1\sin ^2 \theta }{r}  ) B^-  -  \frac{ 1 }{R_1 }\psi_*' E_-\\
& \geq& \bigl( \frac{ \alpha \kappa}{\ve }( \cos ^2 \theta )\psi_*  -  \frac{ 1}{R_1 }\psi_*' \bigr) E_- 
\\
\\&=& \bigl(  \alpha \kappa (R_1\cos  \theta )^2\psi_*  -  \ve  R_1 \psi_*' \bigr) \frac{E_-}{\ve R_1^2}.
\end{eqnarray*}
Hence, using the properties (\ref{cut-off-before}) of the cut-off function $\psi _*(\theta)$, with the zero order coefficient such that $C_* <  \bigl( \alpha R_1 \cos \theta _*(1-\kappa)\bigr) ^2\kappa$, we have that,  for any $0 < \kappa < 1, a > \alpha$ and $\ve $ sufficiently small, 
\begin{eqnarray*}
L B^-(r,\theta ) &\geq&  \frac{-\ve ^2\psi _*'' - \Vert a \Vert R_1 \ve \psi _* '+ \bigl(\alpha R_1 \cos \theta (1-\kappa)\bigr) ^2\kappa  \psi _*}{\ve R_1^2}   E_-\\
&\geq&  \frac{ \bigl(\alpha R_1 \cos \theta _* (1-\kappa)\bigr) ^2\kappa -C_*}{\ve R_1^2}   \psi _* E_-  \geq 0.  
\end{eqnarray*}
Hence, the function $B^-$ is indeed a barrier function for $w^-$. 

(ii) Let us next consider the region where $ \cos \theta > 0$. The cut-off function  \[ \psi ^*(\theta ): [0,2\pi] \rightarrow [0,1] \] is constructed so that it has the following properties:
\begin{subequations}\label{cut-off}
\begin{eqnarray}
0 \leq \psi ^*(\theta ) \leq 1, \quad
\vert (\psi ^*)'' (\theta ) \vert \leq C \vert \psi ^* (\theta) \vert, &&  \hbox{for all} \ \theta \in [0,2\pi);\\ 
\psi ^*(\pi-\theta) = \psi ^*(\pi +\theta),&& \hbox{for all} \ \theta \in (0,\pi];\\
\psi ^* \equiv 1,  \theta \in    (0,\theta^*), &&
\psi  ^* \equiv 0, \ \theta \in    (\mu\theta ^*, \pi ],\\
(\psi^* )'(\theta ) \leq 0,&&  \hbox{for all} \ \theta \in [0,\frac{\pi}{2}). 
\end{eqnarray}
\end{subequations}
Note that this cut-off function is defined independently of $\ve$ (unlike the cut-off function $\psi _*$). 
By symmetry considerations, we will again confine the discussion to the region where $\sin \theta \geq 0$. 

Consider the following  preliminary barrier function 
\[
B^+(r,\theta ) := \psi ^*(\theta) E_+(r,\theta) , 0< \kappa <1, \ E_+(r,\theta):= e^{-\frac{\alpha \kappa \cos (\theta ) (R_2-r)}{\ve}}.
\]
Note the following expressions for the derivatives of the barrier function $B^+$:
\begin{eqnarray*}
B^+_\theta &=& \frac{\alpha \kappa \sin (\theta ) }{\ve} (R_2-r) B^+ + (\psi ^*)'(\theta) E_+; \\
B^+_{\theta \theta} 
&=& \bigl( \alpha \frac{\kappa \cos (\theta) }{\ve} (R_2-r) + \frac{\alpha ^2 \kappa ^2 \sin ^2 (\theta) }{\ve ^2} (R_2-r)^2 \bigr) B^+; \\
&& \quad + ((\psi ^*)''(\theta) +2(\psi ^*)'(\theta ) \frac{\alpha \kappa \sin(\theta) }{\ve} (R_2-r)) 
E_+; \\
B^+_r &=& \frac{\alpha \kappa \cos (\theta) }{\ve}  B^+,   \quad
B^+_{rr} = \frac{\alpha ^2\kappa ^2 \cos ^2 (\theta) }{\ve ^2}  B^+.
\end{eqnarray*}
Observe that,  $\cos ^2 \theta \geq \cos ^2 \theta ^* >0, \ 0 \leq \theta \leq \theta ^* < \pi/2$, and
\begin{eqnarray*}
&&-\frac{\ve}{r^2} B^+_{\theta \theta} -\ve B^+_{rr}+ \alpha (\cos  (\theta)  -\frac{\ve}{\alpha r})B^+_r -\alpha \frac{\sin (\theta)}{r} B^+_\theta \\
&& \geq  
\frac{\alpha ^2\kappa }{\ve} \Bigl( \cos ^2 (\theta) (1-\kappa )-\sin ^2 (\theta)\frac{(R_2-r)}{r} (1+\kappa \frac{(R_2-r)}{r})  + O(\ve ) \Bigr) B^+
\\
&& \geq 
\frac{\alpha ^2\kappa }{\ve} \Bigl( \cos ^2 (\theta)-\sin ^2 (\theta)\frac{(R_2-r)}{r} -\kappa \bigl(\cos ^2 (\theta) + \sin ^2 (\theta)\frac{(R_2-r)^2}{r^2} + O(\ve )\Bigr) B^+
\\
 &&= 
\frac{\alpha ^2\kappa }{\ve} \cos ^2 (\theta)\Bigl(1 -\tan ^2 (\theta)A(r) -\kappa \bigl(1 + \tan ^2 (\theta)A^2)  + O(\ve ) \Bigr) B^+, \\
&& \hbox{where} \quad  A(r) := \frac{R_2}{r} -1 \geq 0, \\
 &&\geq  
\frac{\alpha ^2\kappa }{\ve} \cos ^2 (\theta)\Bigl(1 -\kappa - \tan ^2 (\theta ^*)A\bigl(1 + \kappa A  )+ O(\ve) \Bigr) B^+, \ \forall \theta \leq \theta _*, \\
 &&=  
\frac{\alpha ^2\kappa }{\ve} \cos ^2 (\theta)\Bigl(\frac{1 -\kappa}{\tan ^2 (\theta ^*)} - A\bigl(1 + \kappa A  )\Bigr)\tan ^2 (\theta ^*) + O(\ve ) \Bigr) B^+.
\end{eqnarray*}
Moreover,
\[
(a-\alpha) ((\cos \theta )B^+_r- \frac{\sin \theta}{r}B^+_\theta) \geq 
 \frac{(a-\alpha) \alpha \kappa}{\ve }(1- \frac{R_2  }{r}\sin ^2 \theta   ) B^+ . 
\]
Note that,
\[
A'(r) <0, \ A''(r) > 0, \quad \hbox{and} \quad \frac{R_2-R_1}{R_1} \geq A(r) \geq 0. \]
Observe further that
\begin{subequations}\label{star}
\begin{equation}
\hbox{if} \quad \kappa \leq 0.5\ \hbox{and} \  \frac{R_2}{r} < \frac{1}{ \sin(\theta ^*)} .
\end{equation}
then
\begin{eqnarray*}
\frac{1 -\kappa}{\tan ^2 (\theta ^*)} - A(1 + \kappa A) &\geq & \frac{1}{2} \bigl( \frac{1}{\tan ^2 (\theta ^*)} - (A^2+2A)\bigr) \\
 &= & \frac{1}{2}\Bigl( \bigl( \frac{1}{ \sin (\theta ^*)} - (1+A)\bigr) \bigl( \frac{1}{ \sin (\theta ^*) } + (1+A)\bigr)\Bigr)>0.
\end{eqnarray*}
Consider the case when $r > R_2 \sin  \theta ^*$ (which corresponds to $A< \frac{1}{\sin \theta ^*}-1$) then for all $\theta \in (0, \theta ^*),   0 < \theta ^* < \pi/2$
\begin{equation}
(1- \frac{R_2  }{r}\sin ^2 \theta   ) \geq (1- \sin  \theta ^* ) >0.
\end{equation}
\end{subequations}
Hence, the function $B^+$ is a barrier function in the sub-region where $r > R_2 \sin  \theta ^*, \cos \theta >0, \ \theta \in (0, \theta ^*)$. Note that the inequalities (\ref{star}) are strict inequalities. Hence, we can enlarge the sub-region  so as to include $\theta \in (0,  \mu\theta ^*)$.

We next  consider the region $R_1 \leq r \leq R_2 \sin \mu \theta ^*, \ \theta \in (0, \mu \theta ^*)$.
Note first that 
\[
L(r\psi ^* (\theta)\cos \theta ) =  a \psi ^* -(a \cos \theta -2\ve r^{-1}) \sin \theta \psi ^* _\theta - \ve r^{-1}\cos \theta (\psi ^*) _{\theta \theta}  > \alpha \psi ^*, \]
and
\begin{eqnarray*}  \vert B^+(r,\theta ) \vert \leq  \psi ^* (\theta) e^{-\frac{R_2\alpha \kappa(1- \sin (\mu \theta ^*) \cos  (\mu \theta ^*)}{\ve}}
\leq  \psi ^* (\theta) e^{\frac{-R_2\gamma _1}{2\ve}},\quad
\gamma _1 := \alpha  \cos  (\mu \theta ^*) (1-\sin (\mu \theta ^*)).
\end{eqnarray*}
Use the following composite  barrier function
\[
B^+(r,\theta) +  C e^{\frac{-R_2\gamma _1}{2\ve}}(\psi ^* (\theta) r\cos \theta ) 
\]
to bound the boundary layer function $w^+$. 

(iii) From the crude bounds (\ref{crude}) on the derivatives, we can  establish the  bounds \[ 
\Bigl \Vert \frac{\partial ^{i+j} w}{\partial  r^i \partial \theta ^j} \Bigr \Vert  \leq C \ve ^{-i-j}.
\]
We next  improve on these bounds in the angular direction. 
Consider the following expansions of the layer components:
\begin{eqnarray*}
w^-(r,\theta) = w^-(R_1,\theta)e^{\frac{a(R_1,\theta) \cos (\theta ) (r-R_1)}{\ve}}  +\ve  z^-(r,\theta),\quad \cos \theta <0;\\
 w^+(r,\theta) = w^+(R_2,\theta)e^{-\frac{a(R_2,\theta) \cos (\theta ) (R_2-r)}{\ve}} +\ve  z^+(r,\theta),\quad \cos \theta > 0,
\end{eqnarray*}
where by virtue of assumption (\ref{assump1}), the inequalities
\[ \vert w^-(R_1,\theta) \vert \leq C \vert \cos \theta \vert   \quad \hbox{and} \quad \vert w^+(R_2,\theta) \vert \leq C \vert \cos \theta \vert \] 
follow. Note that
\begin{eqnarray*}
\ve L_{r,\theta} z^- &=& - L_{r,\theta} \bigl( w^-(R_1,\theta)e^{\frac{a(R_1,\theta) \cos (\theta ) (r-R_1)}{\ve}} \bigr), \quad \hbox{and} \quad z^-(R_1,\theta) =0 \\
\ve L_{r,\theta} z^+ &=& - L_{r,\theta} \bigl( w^+(R_2,\theta)e^{-\frac{a(R_2,\theta) \cos (\theta ) (R_2-r)}{\ve}} \bigr), \quad \hbox{and} \quad z^+(R_2,\theta) =0.
\end{eqnarray*}
As in \cite{circle}, one can check that 
\[
\vert L_{r,\theta}  z^- \vert \leq C\frac{1}{\ve }e^{-\frac{a(R_1,\theta) \cos (\theta ) (R_1-r)}{2\ve}} .
\] 
Repeating the argument in (i),  we can then establish that
$
\vert z ^-\vert \leq CB^-(r,\theta), 
$
where $B^-(r,\theta)$ is defined in (\ref{barrier}). 
Using the fact that $\vert w^-(R_1,\theta) \vert \leq C \vert \cos \theta \vert $, the improved bounds on the derivatives in the angular direction for $w^-$ follow. An analogous argument is used for $w^+$.  

Given the construction of the cut-off functions $\psi_*(\theta)$ and $\psi^*(\theta)$, one can check that $w\equiv 0$ in fixed neighbourhoods of the characteristic points. 

\end{document}